\documentclass[11pt]{article}
\usepackage{amsmath, amsthm, amscd, amsfonts, amssymb}
\date{ }

\newcommand{\ga}{\Gamma}
\newcommand{\se}{1\lhd H\lhd K\lhd G}

\newcommand{\al}{\alpha}

\newtheorem{theorem}{Theorem}[section]
\newtheorem{lemma}[theorem]{Lemma}

\title{\bf On The UNRECOGNIZABILITY BY PRIME GRAPH FOR THE ALMOST SIMPLE GROUP ${\rm {\bf PGL(2,9)}}$}
\smallskip
\normalsize
\author{{\bf Ali Mahmoudifar}\\
Department of Mathematics, Tehran North Branch,\\ Islamic Azad University, Tehran, Iran\\ e-mail: alimahmoudifar@gmail.com}
\begin{document}
\maketitle
\begin{abstract}
The prime graph of a finite group $G$ is denoted by
$\ga(G)$.
Also $G$ is called recognizable by prime graph if and only if each finite group $H$
with $\ga(H)=\ga(G)$, is isomorphic to $G$. In this paper, we classify all finite groups with the same prime graph as $\textrm{PGL}(2,9)$. In particular, we present some solvable groups with the same prime graph as $\textrm{PGL}(2,9)$.
\end{abstract}
\section{Introduction}
Let $n$ be a natural number.
We denote by $\pi(n)$, the set of all prime divisors of $n$.
Also Let $G$ be a finite group. The set $\pi(|G|)$ is denoted by
$\pi(G)$. The set of element orders of $G$ is denoted by
$\pi_e(G)$. We denote by $\mu(S)$, the maximal numbers of $\pi_e(G)$ under the divisibility
relation. The \textit{prime graph} of $G$ is a graph whose vertex
set is $\pi(G)$ and two distinct primes $p$ and $q$ are joined by an
edge (and we write $p\sim q$), whenever $G$ contains an element of
order $pq$. The prime graph of $G$ is denoted by $\ga(G)$.
A finite group $G$ is called
\textit{recognizable by prime graph} if for every finite group $H$ such that $\ga(G) = \ga(H)$, then we have
$G\cong H$. So $G$ is \textit{recognizable by prime graph} whenever there exists a fin finite group $K$ such that $\ga(K)=\ga(G)$ in while $K$ is not isomorphic to $G$.

For the almost simple group ${\rm PGL}(2,q)$, there are a lot of results about the recognition by prime graph. In \cite{13}, it is proved that if
$p$ is a prime number which is not a Mersenne or Fermat prime and $p \not= 11$, $19$ and
$\ga(G) = \ga(\textrm{PGL}(2, p))$, then $G$ has a unique nonabelian composition factor which is
isomorphic to $\textrm{PSL}(2, p)$ and if $p = 13$, then $G$ has a unique nonabelian composition
factor which is isomorphic to $\textrm{PSL}(2,13)$ or $\textrm{PSL}(2,27)$.
We know that ${\rm PGL}(2,2^{\al})\cong {\rm PSL}(2,2^{\al})$. For the characterization of such simple groups we refer to \cite{15,16}. In \cite{pgl}, it is proved that if $q = p^{\al}$, where $p$ is an odd prime and $\al$ is an odd natural number, then $\textrm{PGL}(2, q)$ is uniquely determined by its prime graph.

By the above description, we get that the characterization by prime graph of $\textrm{PGL}(2, p^k)$, where $p$ is an odd prime number and $k$ is even, is an open problem.
In this paper as the main result we consider the recognition by prime graph of the almost simple groups $\textrm{PGL}(2,3^2)$. Moreover, we construct some solvable group with the same prime graph as $\textrm{PGL}(2,3^2)$.
\section{Preliminary Results}
\begin{lemma}\label{maza}
Let $G$ be a finite group and $N \unlhd G$ such that $G/N$ is a
Frobenius group with kernel $F$ and cyclic complement $C$. If
$(|F|,|N|)=1$ and $F$ is not contained in $NC_G(N)/N$, then
$p|C|\in\pi_e(G)$ for some prime divisor $p$ of $|N|$.
\end{lemma}
\begin{lemma}\label{fro}
Let $G$ be a Frobenius group with kernel $F$ and complement
$C$. Then the following assertions hold:

(a) $F$ is a nilpotent group.

(b) $|F| \equiv 1 \pmod{|C|}$.

(c) Every subgroup of $C$ of order $pq$, with $p$, $q$ (not necessarily distinct)
primes, is cyclic. In particular, every Sylow subgroup of $C$ of odd order is
cyclic and a Sylow $2$-subgroup of $C$ is either cyclic or generalized quaternion
group. If $C$ is a non-solvable group, then $C$ has a subgroup of index at most
$2$ isomorphic to $SL(2,5) \times M$, where $M$ has cyclic Sylow $p$-subgroups and
$(|M|, 30) = 1$.
\end{lemma}
By using \cite[Theorem A]{21} we have the following result:
\begin{lemma}\label{wil}
Let $G$ be a finite group with $t(G)\geq 2$. Then
one of the following holds:

(a) $G$ is a Frobenius or 2-Frobenius group;

(b) there exists a nonabelian simple group $S$ such that $S\leq$
$\overline{G}$$:=G/N \leq \textrm{Aut}(S)$ for some nilpotent normal
$\pi_1$-subgroup $N$ of $G$ and $\overline{G}/S$ is a $\pi_1$-group.
\end{lemma}
\section{Main Results}
\begin{lemma}\label{mainl1}
There exists a Frobenius group $G=K:C$, where $K$ is an abelian $3$-group and $\pi(C)=\{2,5\}$,
such that $\ga(G)=\ga({\rm PGL}(2,9))$.
\end{lemma}
\begin{proof}
Let $F$ be a finite field with $3^4$ elements. Also let $V$ be the additive group of $F$ and $H$ be the multiplicative group $F\setminus\{0\}$. We know that $H$ acts on $V$ by right product. So $G:=V\rtimes H$ is a finite group such that $\pi(G)=\{2,3,5\}$, since $|V|=3^4$ and $|H|=80$. On the other hand $H$ acts fixed point freely on $V$, so $G$ is a Frobenius group with kernel $V$ and complement $H$. Since the multiplicative group $F\setminus\{0\}$ is cyclic, $H$ is cyclic too. Therefore, $G$ has an element of order $10$, and so the prime graph of $G$ consists just one edge, which is the edge between $2$ and $5$. This implies that $\ga(G)=\ga({\rm PGL}(2,9))$, as desired.
\end{proof}
\begin{lemma}\label{mainl2}
There exists a Frobenius group $G=K:C$, where $\pi(K)=\{2,5\}$ and $C$ is a cyclic $3$-group
such that $\ga(G)=\ga({\rm PGL}(2,9))$.
\end{lemma}
\begin{proof}
Let $F_1$ and $F_2$ be two fields with $2^2$ and $5^2$ elements, respectively. Let $V:=F_1\times F_2$ be the direct product of the additive groups $F_1$ and $F_2$. Also let $H:=H_1\times H_2$, be the direct product of $H_1$ and $H_2$, which are the multiplicative groups $F_1\setminus\{0\}$ and $F_2\setminus\{0\}$, respectively. We know that $H_i$ acts fixed point freely on $F_i$, where $1\leq i\leq2$. So we define an action of $H$ on $V$ as follows: for each $(h,h')\in H$ and $(g,g')\in V$, we define $(g,g')^{(h,h')}:=(hg,h'g')$. It is clear that this definition is well defined. So we may construct a finite group $G=V\rtimes H$. On the other hand $H$ acts fixed point freely on $V$. So $G$ is a Frobenius group with kernel $V$ and complement $H$ such that $\pi(V)=\{2,5\}$ and $\pi(H)=\{3\}$. Finally, since $V$ is nilpotent, we get that $\ga(G)$ contains an edge between $2$ and $5$, so $\ga(G)=\ga({\rm PGL}(2,9))$.
\end{proof}
\begin{lemma}\label{mainl3}
There exists a $2$-Frobenius group $G$ with normal series $\se$, such that $\pi(H)=\{5\}$, $\pi(G/K)=\{2\}$ and $\pi(K/H)=\{3\}$,
such that $\ga(G)=\ga({\rm PGL}(2,9))$.
\end{lemma}
\begin{proof}
Let $F$ be a field with $5^2$ elements and $V$ be its additive group. We know that $F=\{0,1,\alpha, \alpha^2,\dots, \alpha^{23}\}$, where $\alpha$ is a generator of the multiplicative group $F\setminus\{0\}$. Hence $|\alpha|=24$, and so $\beta:=\alpha^{8}$ has order $3$.
Also $\langle \beta\rangle\cong Z_3$ and ${\rm Aut}(Z_3)\cong Z_2$. This argument implies that we may construct a Frobenius group $T:=\langle \beta\rangle:\langle\gamma\rangle$, where $\gamma$ is an involution.

Now we define an action of $T$ on $V$ as follows: for each $\beta^x\gamma^y\in T$ and $v\in V$,
$v^{\beta^x\gamma^y}:=\beta^x v$, where $1\leq x\leq3$ and $1\leq y\leq2$. Therefore $G:=V:T$ is a $2$ Frobenius group with desired properties.
\end{proof}
\begin{theorem}\label{th1}
Let $G$ be a finite group. Then $\ga(G)=\ga({\rm PGL}(2,3^2))$ if and only if $G$ is isomorphic to one of the following groups:

(1) A Frobenius group $K:C$, where $K$ is an abelian $3$-group and $\pi(C)=\{2,5\}$,

(2) A Frobenius group $K:C$, where $\pi(K)=\{2,5\}$ and $C$ is a cyclic $3$-group,

(3) A $2$-Frobenius group with normal series $\se$, such that $\pi(H)\subseteq\{2,5\}$, $\pi(G/K)=\{2\}$ and $\pi(K/H)=\{3\}$,

(4) Almost simple group ${\rm PGL}(2,3^2)$.
\end{theorem}
\begin{proof}
Throughout the proof, we assume that $G$ is a finite group with the same prime graph as the almost simple group
${\rm PGL}(2,3^2)$. First we note that by \cite[Lemma 7]{mogh-shi}, we have:
$$\mu({\rm PGL}(2,9))=\{3, 8, 10\}.$$

Hence in $\ga({\rm PGL}(2,3^2))$ (and so in $\ga(G)$), there is only one edge which is the edge between $2$ and $5$ and so $3$ is an isolated vertex. This implies that $\ga(G)$ has two connected components $\{3\}$ and $\{2,5\}$. Thus by Lemma~\ref{wil}, we get that $G$ is a Frobenius group or $2$-Frobenius group or there exists a nonabelian simple group $S$ such that $S\leq G/Fit(G)\leq Aut(S)$. We consider each possibility for $G$.

Let $G$ be a Frobenius group with kernel $K$ and complement $C$. We know that $K$ is nilpotent and $C$ is a connected component of the prime graph of $G$. Also by the above description, $3$ is not adjacent to $2$ and $5$ in $\ga(G)$. This shows that either $\pi(K)=\{3\}$ or $\pi(C)=\{3\}$. We consider these cases, separately.

{\bf Case 1.} Let $\pi(K)=\{3\}$. Hence the order of complement $C$, is even and so $K$ is an abelain subgroup of $G$.
Also by the above description, $\pi(C)=\{2,5\}$. Since $C$ is a connected component of $\ga(G)$, there is and edge between $2$ and $5$. So it follows that $\ga(G)=\ga({\rm PGL}(2,9))$, which implies groups satisfying in (1).

{\bf Case 2.} Let $\pi(C)=\{3\}$. Hence $\pi(K)=\{2,5\}$. Since $K$ is nilpotent, we get that $2$ and $5$ are adjacent in $\ga(G)$. This means $\ga(G)=\ga({\rm PGL}(2,9))$ and so we get (2).

{\bf Case 3.} Let $G$ be a $2$-Frobenius group with normal series $\se$. Since $\pi(K/H)$ and $\pi(H)\cup\pi(G/K)$ are the connected components of $\ga(G)$, we get that $\pi(K/H)=\{3\}$ and $\pi(H)\cup\pi(G/K)=\{2,5\}$. This implies (3).

{\bf Case 4.} Let there exist a nonabelian simple group $S$, such that $S\leq \bar{G}:=G/{\rm Fit}(G)\leq {\rm Aut}(S)$. Since
$\pi(S)\subseteq\pi(G)$, $\pi(S)=\{2,3,5\}$. The finite simple groups with this property are classified in \cite[Table 8]{mogh}. So we get that $S$ is isomorphic to one of the simple groups $A_5$, ${\rm PSU}(4,2)$ and ${\rm PSL}(2,9)$($\cong A_6$).

{\bf Subcase 4.1.} Let $S\cong A_5$. We know that ${\rm Aut}(A_5)=S_5$. So $\bar{G}$ is isomorphic to the alternating group $A_5$ or the symmetric group $S_5$. Since in the prime graph of $S_5$, there is an edge between $2$ and $3$, hence we get that $\bar{G}$ is not isomorphic to $S_5$. Thus, $G/{\rm Fit(G)}=A_5$.
In the prime graph of $A_5$, $2$ and $5$ are nonadjacent. Then at least one of the prime numbers $2$ or $5$, belongs to $\pi({\rm Fit}(G))$.

Let $5\in\pi({\rm Fit}(G))$. Let $F_5$ be a Sylow $5$-subgroup of ${\rm Fit}(G)$. Since $F_5$ is a characteristic subgroup of ${\rm Fit}(G)$ and ${\rm Fit}(G)$ is a normal subgroup of $G$, $F_5\unlhd G$.
On the other hand in alternating group $A_5$, the subgroup $\langle (12)(34), (13)(24)\rangle:\langle(123)\rangle$ is a Frobenius subgroup isomorphic to $2^2:3$. We recall that by the previous argument, $F_5\unlhd G$, and so $G$ has a subgroup isomorphic to $5^{\alpha}:(2^2:3)$. So by Lemma~\ref{maza}, we get that $3$ is adjacent to $5$, a contradiction.

{\bf Subcase 4.2.} Let $S\cong {\rm PSU}(4,2)$. By \cite{spec1}, there is an edge between $3$ and $2$ which is a contradiction.

{\bf Subcase 4.3.} Let $S\cong {\rm PSL}(2,9)$. Then $\bar{G}$ is isomorphic to ${\rm PSL}(2,9)$ or ${\rm PSL}(2,9):\langle\theta\rangle$, where $\theta$ is a diagonal, field or diagonal-field automorphism of ${\rm PSL}(2,9)$. In particular $\theta$ is an involution. If $\theta$ is a field or diagonal-field automorphism of ${\rm PSL}(2,9)$, then the semidirect product ${\rm PSL}(2,9):\langle\theta\rangle$ contains an element of order $6$. Hence $\theta$ is neither a field automorphism nor a diagonal-field automorphism. Therefore $\theta$ is a diagonal automorphism and so
$\bar{G}\cong {\rm PGL}(2,9)$.

By the above discussion, $G/{\rm Fit}(G)\cong {\rm PGL}(2,9)$. It is enough to prove that ${\rm Fit}(G)=1$. On the contrary, let $r\in\pi({\rm Fit}(G))$. Also let $F_r$ be the Sylow $r$-subgroup of ${\rm Fit}(G)$. Since ${\rm Fit}(G)$ is nilpotent,
we can write ${\rm Fit}(G)=O_{r'}({\rm Fit}(G))\times F_r$. So if we put $\tilde{G}=G/O_{r'}({\rm Fit}(G))$, then we get that:
$${\rm PGL}(2,9)\cong \frac{G}{{\rm Fit}(G)}\cong\frac{\tilde{G}}{F_r}\cong \frac{\tilde{G}/\Phi(F_r)}{F_r/\Phi(F_r)}.$$
Since $F_r/\Phi(F_r)$ is an elementary abelian group, without loose of generality we may assume that $F:={\rm Fit}(G)$ is an elementary abelian $r$-group and $G/F\cong {\rm PGL}(2,9)$.

If $r=3$, then by , we conclude that in $\ga(G)$, $2$ and $3$ are adjacent, which is a contradiction.
So let $r\not=3$. Also let $S_3$ be a Sylow $3$-subgroup of ${\rm PGL}(2,9)$. We know that $S_3$ is not cyclic.
On the other hand $F\rtimes S_3$ is a Frobenius group since $3$ is an isolated vertex of $\ga(G)$. This follows that $S_3$ is cyclic which is impossible. Therefore $F=1$ and so $G\cong {\rm PGL}(2,9)$, which completes the proof.
\end{proof}

\end{document}